\documentclass[reqno]{amsart}
\usepackage{fullpage}
\usepackage{array}
\usepackage{hyperref}
\usepackage{graphicx,color}
\usepackage{float}
\usepackage{upgreek}

\newcommand{\gen}[1]{\ensuremath{\langle{#1}\rangle}}

\newcommand{\var}{\ensuremath{\textnormal{var}}}

\newcommand{\black}{\color{black}}

\newtheorem{theorem}{Theorem}[section]
\newtheorem{example}[theorem]{Example}

\newtheorem{lemma}[theorem]{Lemma}
\newtheorem{corollary}[theorem]{Corollary}
\newtheorem{remark}[theorem]{Remark}
\newtheorem{proposition}[theorem]{Proposition}

\newcolumntype{C}[1]{>{\centering\let\newline\\\arraybackslash\hspace{0pt}}m{#1}}

\usepackage[utf8]{inputenc}
\usepackage{amsfonts}
\usepackage{amssymb}
\usepackage{verbatim}
\usepackage{mathrsfs}

\numberwithin{equation}{section}

\begin{document}
	\title{On the multiplicities of  the central cocharacter of algebras with polynomial identities}
	
	\author{
Wesley Quaresma Cota$^{a,1}$ and Thais Silva do Nascimento$^{b,*}$ 
}

\thanks{{\it E-mail addresses:} quaresmawesley@gmail.com (Cota) and  thais.nascimento@ufmt.br (do Nascimento).}

\thanks{\footnotesize $^{1}$ Partially supported by FAPESP, grant no. 2025/05699-0.}

\thanks{\footnotesize $^{*}$ Corresponding author.}

\subjclass[2020]{Primary 16R10, 16W50, Secondary 20C30}

\keywords{Polynomial identities, central polynomials, colength sequence, central cocharacter}

\dedicatory{$^{a}$ IME, USP, Rua do Matão 1010, 05508-090, São Paulo, Brazil \\
$^b$IME, UFMT, Avenida Fernando Corrêa da Costa 2367, 78060-900, Cuiabá, Brazil}
	
	\begin{abstract}  

For an associative algebra $A$ over a field of characteristic zero, let $P_n(A)$ and $P_n^z(A)$ denote the spaces of multilinear polynomials of degree $n$ modulo the polynomial identities and the central polynomials of $A$, respectively. We also write $\Delta_n(A)$ for the space of multilinear central polynomials of degree $n$ modulo the polynomial identities of $A$. The corresponding sequences of colengths, central colengths and proper central colengths measure the number of irreducible components in the $S_n$-module decompositions of $P_n(A)$, $P_n^z(A)$ and $\Delta_n(A)$, respectively. In this paper, we investigate several examples of PI-algebras and explicitly describe their cocharacter, central cocharacter and proper central cocharacter sequences. As a consequence, we obtain a complete classification, up to PI-equivalence, of all algebras whose sequences of colengths and central colengths are bounded by a constant.
 \black
	\end{abstract}
	
	\maketitle
	
	\section{Introduction}

Determining the complete set of polynomial identities satisfied by a given algebra is a notoriously difficult problem. For instance, the description of the $\textnormal{T}$-ideal of the algebra of $n \times n$ matrices $M_n(F)$ is known only for $n \leq 2$. In this context, the works of Regev and Kemer have significantly advanced the theory by introducing an approach based on numerical invariants associated with the algebra, such as the codimension and colength sequences. These invariants, which constitute the central focus of this paper, have become fundamental tools in the study of polynomial identities, as they provide valuable information about their asymptotic behavior.  

Let $F$ be a field of characteristic zero and let $A$ be an associative algebra over $F$. Denote by $P_n$ the space of multilinear polynomials of degree $n$ and by $P_n(A)$ the quotient space $$P_n(A)=\frac{P_n}{P_n\cap \textnormal{Id}(A)},$$ where $\textnormal{Id}(A)$ is the $\textnormal{T}$-ideal of polynomial identities of $A$. The number $c_n(A) = \dim_F P_n(A)$ is called the $n$-th codimension of $A$. Regev \cite{RG} proved that an algebra $A$ satisfies a non-zero polynomial identity (i.e., is a PI-algebra) if and only if the sequence $c_n(A), \, n \geq 1$, is exponentially bounded. Furthermore, Kemer \cite{Kem} showed that this sequence is either polynomially bounded or grows exponentially.  

Alongside polynomial identities, a related object of study is the set of {central polynomials} of an algebra $A$. These are polynomials whose evaluations always belong to the center of $A$. If a central polynomial takes a non-zero value in the center of $A$, it is called a {proper central polynomial}. Associated with the $\textnormal{T}$-space $\textnormal{Id}^z(A)$ of central polynomials of $A$, Regev introduced two new sequences of numerical invariants: for $n \geq 1$, the $n$-th {central codimensions} and {proper central codimensions} are defined, respectively,  by $$c_n^z(A)=\dim_F P_n^z(A)\mbox{ and }\delta_n(A)=\dim_F \Delta_n(A),$$
$$\mbox{ where } P_n^z(A) =  \frac{P_n}{P_n \cap \textnormal{Id}^z(A)}
\quad \text{and} \quad
\Delta_n(A) = \frac{P_n \cap \textnormal{Id}^z(A)}{P_n \cap \textnormal{Id}(A)}.
$$

By definition, one has the decomposition
\begin{equation} \label{relacaodascodimensoes}
c_n(A) = c_n^z(A) + \delta_n(A)    
\end{equation}
and hence, if $A$ is a PI-algebra, both the central and proper central codimension sequences are also exponentially bounded. 

It is worth emphasizing that these sequences have been computed explicitly only for a few examples of PI-algebras, and their general behavior is far from being completely understood. Nonetheless, some progress has been made by considering the exponential growth of these sequences. In this direction, Giambruno and Zaicev \cite{AntMik,GiZai} proved that the limits  
\[
\exp(A)=\lim_{n\rightarrow\infty}\sqrt[n]{c_n(A)},\qquad 
\exp^{z}(A)=\lim_{n\rightarrow\infty}\sqrt[n]{c_n^{z}(A)}\qquad 
\text{and}\qquad 
\exp^{\delta}(A)=\lim_{n\rightarrow\infty}\sqrt[n]{\delta_n(A)},
\]
called the PI-exponent, the central exponent and the proper central exponent of a PI-algebra, respectively, always exist and are non-negative integers. Therefore, they provided a positive solution to Amitsur’s conjecture in this setting.

It is well known that the spaces $P_n(A)$, $P_n^z(A)$ and $\Delta_n(A)$ carry a natural action of the symmetric group $S_n$ by permuting their variables. This induces a decomposition of these spaces as $S_n$-modules into direct sums of irreducible modules indexed by partitions $\lambda \vdash n$. Thus, one can decompose their corresponding characters, respectively, as
\[
\chi_n(A) = \sum_{\lambda \vdash n} m_\lambda \chi_\lambda, \quad
\chi_n^z(A) = \sum_{\lambda \vdash n} m_\lambda^z \chi_\lambda \quad \mbox{ and }\quad 
\chi_n^\delta(A) = \sum_{\lambda \vdash n} m_\lambda^\delta \chi_\lambda,
\]
where $\chi_\lambda$ denotes the irreducible character of $S_n$ corresponding to $\lambda$, and $m_\lambda$, $m_\lambda^z$, $m_\lambda^\delta$ are the corresponding multiplicities. The characters $\chi_n(A)$, $\chi_n^z(A)$ and $\chi_n^\delta(A)$ are called the $n$-th {cocharacter}, {central cocharacter} and {proper central cocharacter} of $A$. Clearly, we have $$\chi_n(A) = \chi_n^z(A) + \chi_n^\delta(A)\mbox{ and }m_\lambda =m_\lambda^z+ m_\lambda^\delta .$$

Determining the complete decomposition of these cocharacters is a central and challenging problem in PI-theory. Moreover, recent progress has been achieved in the study of the decomposition of central and proper central cocharacters for the minimal varieties contained in $\textnormal{var}(UT_2)$ and $\textnormal{var}(\mathcal{G})$ (see, for instance, \cite{Costa}).

The sums
\[
l_n(A) = \sum_{\lambda \vdash n} m_\lambda, \quad
l_n^z(A) = \sum_{\lambda \vdash n} m_\lambda^z \quad \mbox{ and }\quad 
l_n^\delta(A) = \sum_{\lambda \vdash n} m_\lambda^\delta
\]
are defined as the $n$-th {colength}, {central colength} and {proper central colength} of $A$, respectively. In \cite{Mishchenko}, Mishchenko, Regev and Zaicev characterized the varieties of polynomial growth as precisely those for which the sequence $l_n(A), \, n \geq 1$, is bounded by a constant. This motivated a new approach to classifying varieties based on their colength sequences, as developed in \cite{GLa, Daniela}.


This paper aims to extend these results by classifying varieties with respect to both their colength and central colength sequences. We emphasize that the classification of varieties with colength bounded by $4$ is entirely contained within subvarieties of the variety generated by the algebra $UT_2$ of upper triangular matrices of size $2$, whose structure is well known (see \cite[Theorem 5.4]{Danielasub}). To establish the result we aim for, it is necessary to study varieties outside $\textnormal{var}(UT_2)$, thereby revealing an even greater level of complexity when dealing with higher colengths.  

It is worth mentioning that colength sequences have also been studied for algebras with additional structures, see \cite{Cota2, Cota, MN, NV, Ana} for further details. 



\section{Polynomial identities and central polynomials}

Let \( F\langle X \rangle \) be the free associative algebra generated by a countable set \( X=\{x_1, x_2, \ldots \} \) of variables over a field \( F \) of characteristic zero. 

We say that \( A \) is a PI-algebra if it satisfies a nontrivial polynomial identity, that is, a nonzero polynomial $f$ that vanishes under all evaluations in \( A \). In this case, we denote $f\equiv 0$ on $A$. Associated to the algebra \( A \), we define
\[
\textnormal{Id}(A) = \{f \in F\langle X \rangle \mid f \equiv 0 \text{ on } A\},\quad \textnormal{Id}^z(A) = \{f(x_1 ,\ldots, x_n) \in F\langle X \rangle \mid f(a_1, \ldots, a_n)\in Z(A)\}
\] $$ \mbox{ and } \textnormal{Id}^z(A)- \textnormal{Id}(A)$$
the sets of all polynomial identities, central polynomials and proper central polynomials of \( A \), respectively, where $Z(A)$ denotes the center of $A$. These sets form $\textnormal{T}$-spaces of \( F\langle X \rangle \), that is, subspaces invariant under all endomorphisms of the free algebra. Moreover, $\textnormal{Id}(A)$ is an ideal of the free algebra, and hence is referred to as the T-ideal of $A$.

Since we are interested in an investigation from the point of view of polynomial identities, we define the variety generated by \( A \), denoted by \( \mathcal{V} = \textnormal{var}(A) \), as the class of all algebras satisfying all the identities of \( A \). If two algebras \( A \) and \( B \) have the same $\textnormal{T}$-ideal, then we say that they are \(\textnormal{T}\)-equivalent and write \( A \sim_T B \).

In characteristic zero, it is well known that \( \textnormal{Id}(A) \), \( \textnormal{Id}^z(A) \) and $\textnormal{Id}^z(A)- \textnormal{Id}(A)$ are completely determined by their multilinear polynomials. For each \( n \geq 1 \), we define \( P_n \) as the vector space of multilinear polynomials of degree \( n \) in the variables \( x_1, \ldots, x_n \), that is,
\[
P_n = \mathrm{span}_F\{x_{\sigma(1)} \cdots x_{\sigma(n)} \mid \sigma \in S_n\}.
\]
To study the asymptotic behavior of the polynomial identities, central polynomials and proper central polynomials of \( A \), Regev~\cite{RG} introduced three sequences: the sequence of codimensions \( \{c_n(A)\}_{n \geq 1} \), central codimensions \( \{c_n^z(A)\}_{n \geq 1} \) and proper central codimensions \( \{\delta_n(A)\}_{n \geq 1} \), where their $n$-th terms are defined, respectively, by $$c_n(A) = \dim_F P_n(A),\, c_n^z(A) = \dim_F P_n^z(A)\mbox{ and }\delta_n (A) = \dim_F \Delta_n(A),$$
$$ \mbox{ where }\, P_n(A)=  \frac{P_n}{P_n \cap \textnormal{Id}(A)}, \quad P_n^z(A)=  \frac{P_n}{P_n \cap \textnormal{Id}^z(A)}\quad 
\mbox{ and }\quad  \Delta_n(A)=  \frac{P_n\cap \textnormal{Id}^z(A)}{P_n \cap \textnormal{Id}(A)}.$$

In what follows, the superscript ${T}$ indicates that the corresponding set generates the subspace as a $\textnormal{T}$-space, whereas the subscript $T$ denotes that they generate it as a $\textnormal{T}$-ideal.

\begin{example} \cite{Malvec, Regev}
Let \( \mathcal{G}=
\langle 1,e_1,e_2,\ldots \mid e_ie_j=-e_je_i\rangle  \) be the infinite-dimensional Grassmann algebra, and let \( UT_2 \) denote the algebra of \( 2 \times 2 \) upper triangular matrices over \( F \). Then,
\[
\textnormal{Id}(UT_2) =\textnormal{Id}^z(UT_2)= \langle [x_1, x_2][x_3, x_4] \rangle_T, \quad  \textnormal{Id}(\mathcal{G}) = \langle [x_1, x_3, x_4] \rangle_T \quad \text{and} \quad \textnormal{Id}^z(\mathcal{G})= \langle [x_1,x_2], \textnormal{Id}(\mathcal{G})\rangle^T.
\]
Moreover, $c_n(\mathcal{G}) = 2^{n-1}$, $c_n^z(\mathcal{G})= \delta_n (\mathcal{G})= 2^{n-2}$, $c_n(UT_2)=c_n^z(UT_2) = 2^{n-1}(n-2) + 2$ and $\delta_n(UT_2)=0$.
\end{example}

 In \cite{RG}, Regev proved that if $A$ is a PI-algebra over a field of characteristic zero, then there exist constants $\alpha, \beta \geq 0$ such that  
\[
c_n(A) \leq \beta\, \alpha^{\,n} ,\quad \text{for all } n \geq 1.
\]
Hence, by (\ref{relacaodascodimensoes}), both sequences $c_n^z(A)$ and $\delta_n(A)$, $n \geq 1$, are also exponentially bounded.

This result initiated a systematic study of varieties of polynomial growth.  Recall that an algebra $A$ is said to have {polynomial codimension growth} if there exist constants $\alpha, t \geq 0$ such that  
\[
c_n(A) \leq \alpha n^{t}, \quad \text{for all } n \geq 1.
\]

Below, we gather some well-known characterizations of varieties of polynomial growth.

\begin{theorem}\cite{ GiamZai, Kem} \label{teokemer}
    Let $A$ be an algebra over a field $F$ of characteristic zero. Then $A$ has polynomial growth if and only if one of the following cases occurs
\begin{enumerate}
    \item[1)] $\mathcal{G},UT_2\notin \textnormal{var}(A).$

    \item[2)] $A\sim_T B_1\oplus \cdots \oplus B_m$, where each $B_i$ is a finite-dimensional algebra over $F$ such that $\dim B_i/J(B_i) \leq 1$ where $J(B_i)$ denotes its Jacobson radical. 
\end{enumerate}

\end{theorem}
As a consequence, the varieties generated by $\mathcal{G}$ and $UT_2$ are the only varieties of almost polynomial growth, that is, they have exponential codimension growth, while every proper subvariety has polynomial growth. 

It is worth noting that, by \cite[Theorem~3]{AntMik}, an algebra $A$ has polynomial growth of the codimension sequence $\{c_n(A)\}_{n \geq 1}$ if and only if the central codimension sequence $\{c_n^z(A)\}_{n \geq 1}$ also has polynomial growth. Consequently, the varieties $\mathrm{var}(UT_2)$ and $\mathrm{var}(\mathcal{G})$ are precisely those with almost polynomial growth of the sequence of central codimensions. More recently, in \cite{GiamLaMilies}, the authors classified the varieties of almost polynomial growth with respect to the proper central codimensions. Consequently, the sequences $c_n(A)$, $c_n^z(A)$ and $\delta_n(A)$, $n \geq 1$, are either polynomially bounded or grow exponentially, with no intermediate growth possible.

Motivated by the previous results, La Mattina \cite{Danielasub} classified the subvarieties of the varieties of almost polynomial growth. In order to state this result, let us define the following algebras. 

For $m\geq 2$, let $UT_m$ be the algebra of $m\times m$ upper triangular matrices and consider the elements $E = \sum\limits_{i = 1}^{m-1} e_{i,i+1}$ and $I_m=\sum\limits_{i = 1}^{m-1}e_{ii}$ the identity matrix of order $m$, where $e_{ij}$ denotes the elementary matrix with $1$ in the $(i,j)$-entry and zeros elsewhere. Define the following subalgebras of $UT_m$
\begin{align*}
A_{m,1} \; &=\; \textnormal{span}_F\{e_{11}, E, \ldots , E^{m-2}; e_{12}, e_{13}, \ldots , e_{1m}\}, \\[4pt]
A_{m,1}^* \; &=\; \textnormal{span}_F\{e_{mm}, E, \ldots , E^{m-2}; e_{1m}, e_{2m}, \ldots , e_{m-1\, m}\}, \\[4pt]
N_m \; \; \; &=\; \textnormal{span}_F\{I_m, E, \ldots , E^{m-2}; e_{12}, e_{13}, \ldots , e_{1m}\}.
\end{align*}

Moreover, for $k\geq 2$, we denote by $\mathcal{G}_{2k}=\langle 1,e_1, \ldots, e_{2k}\mid e_ie_j=-e_je_i \rangle$ as the subalgebra of the infinite-dimensional Grassmann algebra generated by $1,e_1, \ldots, e_{2k}$. It is straightforward to verify that $\mathcal{G}_2\sim_T N_3$. 

\begin{theorem} \label{subvarieties} \cite[Theorems 5.2 and 5.4]{Danielasub}  Let $A$ be an algebra over a field of characteristic zero.
\begin{enumerate}
    \item[1)] If $A\in \textnormal{var}(\mathcal{G})$ then $A$ is PI-equivalent to a finite direct sum of algebras in the set  $$\{N, C, \mathcal{G}_{2k},\mathcal{G}\mid k\geq 1\}.$$ 

    \item[2)] If $A\in \textnormal{var}(UT_2)$ then $A$ is PI-equivalent to a finite direct sum of algebras in the set $$\{N, C, N_{r},A_{m,1}, A_{m,1}^*,UT_2\mid r\geq 3,m\geq 2\}$$
\end{enumerate} where $N$ denotes a nilpotent algebra and $C$ a commutative non-nilpotent algebra.    
\end{theorem}

It is well known that the symmetric group \( S_n \) acts naturally on the space \( P_n \) by permuting the variables \( x_1, \ldots, x_n \). Since $P_n\cap \textnormal{Id}(A)$ and $P_n\cap \textnormal{Id}^z(A)$ are invariant under this action, the quotient spaces \( P_n(A) \), $P_n^z(A)$ and $\Delta_n(A)$ inherit a natural structure of \( S_n \)-modules. We denote by \( \chi_n(A) \), $\chi_n^z(A)$ and $\chi_n^\delta(A)$ the characters associated to the $S_n$-modules \( P_n(A) \), \( P_n^z(A) \) and \( \Delta_n(A) \), respectively, which are called the \( n \)-th cocharacter, central cocharacter and proper central cocharacter of \( A \).

Since we are working over a field \( F \) of characteristic zero, the representation theory of symmetric groups implies that \( P_n(A) \), \( P_n^z(A) \) and \( \Delta_n(A) \) admit a decomposition into irreducible \( S_n \)-modules labeled by partitions \( \lambda \vdash n \). Therefore, we can decompose
\begin{equation} \label{character}
\chi_n(A) = \sum_{\lambda \vdash n} m_\lambda \chi_\lambda, \quad  \chi_n^z(A) = \sum_{\lambda \vdash n} m_\lambda^z \chi_\lambda \quad \mbox{ and }\quad \chi_n^\delta(A) = \sum_{\lambda \vdash n} m_\lambda ^\delta\chi_\lambda
\end{equation} where \( \chi_\lambda \) denotes the irreducible character corresponding to the partition \( \lambda \) and \( m_\lambda \), \( m_\lambda^z \) and \( m_\lambda^
\delta\) denote the corresponding multiplicity. Then, $$\chi_n(A)=\chi_n^z(A)+\chi_n^\delta(A)\quad \mbox{ and }\quad m_\lambda = m_{\lambda}^z+m_{\lambda}^\delta.$$

The \( n \)-th colength, central colength and proper central colength of \( A \) are defined, respectively, by
\[
l_n(A) = \sum_{\lambda \vdash n} m_\lambda, \quad l_n^z(A) = \sum_{\lambda \vdash n} m_\lambda^z \quad \mbox{ and }\quad l_n^\delta(A) = \sum_{\lambda \vdash n} m_\lambda^\delta.
\]

There is a strong connection between the codimension sequence and the colength sequence. In fact, Mishchenko, Regev and Zaicev~\cite{Mishchenko} proved that a variety $\mathcal{V}$ has polynomial codimension growth if and only if there exists a constant $k \geq 0$ such that  
\[
l_n(A) \leq k,\quad \text{for all } n \geq 1.
\]

This characterization of varieties with polynomial growth has motivated further investigations into the classification of varieties according to the behavior of the sequence $\{l_n(A)\}_{n \geq 1}$. In this direction, important contributions were made by Giambruno and La Mattina~\cite{GLa, Daniela}, who classified the varieties whose colength sequences are bounded by $4$.

\section{Constructing algebras with small colength}

In this section, we introduce several algebras that play a fundamental role in the results presented later in this paper. For the algebras under consideration, we describe some $\textnormal{T}$-ideals, codimensions, cocharacters and compute the colengths sequences. These computations will be used in the classification results developed in the subsequent sections.

Before we introduce our extensive list of algebras, we present a useful tool to compute the multiplicities $m_{\lambda}$ in the decomposition of $\chi_n(A)$. For this purpose, we need to resort to the representation theory of the general linear group $GL_m$. 

Let $F_m := F\langle x_{1}, \ldots ,x_{m}\rangle$
be the free associative algebra of rank $m$ and define the subspace $U_m:= \mathrm{span}_F\{x_{1}, \ldots, x_{m}\}$. Observe that the general linear group $GL(U_m) \cong GL_m$ acts naturally on $U_m$, and this action extends diagonally to an action on the subspace $F_{m}^n$ of $F_m$ consisting of all polynomials of degree $n$.

Since $F_{m}^n \cap \textnormal{Id}(A)$ is invariant under this action, the quotient space
\[
F_{m}^n(A) := \frac{F_{m}^n}{F_{m}^n \cap \textnormal{Id}(A)}
\]
inherits a natural structure of a $GL_m$-module. The character of this module is denoted by $\psi_n ^m(A)$.

Over a field $F$ of characteristic zero, the irreducible polynomial representations of $GL_m$ are indexed by partitions $\lambda \vdash n$ with height $h(\lambda)$ of the corresponding diagram is at most $ m$. Hence, we can decompose the character $\psi_n ^m(A)$ as follows:
\begin{equation}\label{eq:char_decomp}
\psi_n ^m(A) = \sum_{\substack{\lambda \vdash n \\ h(\lambda)\leq m}} \widetilde{m}_\lambda \psi_\lambda,
\end{equation}
where $\psi_\lambda$ is the irreducible character corresponding to the partition $\lambda$ and $\widetilde{m}_\lambda$ denotes its multiplicity.

\begin{proposition} \cite{Berele, Drensky2}
Let $\chi_n (A)$ and $\psi_n^m (A)$ denote the $n$-th cocharacter and the $GL_m$-character of $A$, respectively, as defined in equations \eqref{character} and \eqref{eq:char_decomp}. Then, $m_\lambda = \widetilde{m}_\lambda$ for every partition $\lambda \vdash n$ with $h(\lambda) \leq m$.
\end{proposition}

According to \cite[Theorem 12.4.12]{Drensky}, each irreducible $GL_m$-module is cyclic and it is generated by a highest weight vector, denoted $f_\lambda$, associated to a partition $\lambda \vdash n$. Moreover, $f_\lambda$ can be expressed as a linear combination of certain polynomials $f_{T_\lambda}$ constructed as follows (see \cite[Proposition 12.4.14]{Drensky}).

Given a partition $\lambda \vdash n$, the initial Young tableau of shape $\lambda$ is defined as the one filled with the integers $1$ to $n$ in increasing order from top to bottom, column by column. For a Young tableau $T_{\lambda}$ of shape $\lambda$, let $\sigma \in S_n$ be the unique permutation that transforms the initial multitableau into $T_\lambda$. Then the corresponding highest weight vector $f_{T_\lambda}$ is defined by
\[
f_{T_\lambda} = \left(\prod_{i=1}^{\lambda_1} St_{h_i(\lambda)}(x_1, \ldots, x_{h_i(\lambda)})\right)\sigma^{-1},
\]
where $St_t(x_1,\ldots,x_t)$ denotes the standard polynomial of degree $t$, $h_i(\lambda)$ is the height of the $i$-th column of the Young diagram $\lambda$, and the symmetric group $S_n$ acts on the right by place permutation.

As a consequence, we can now state a key result that provides a practical criterion for computing the multiplicities $\widetilde{m}_\lambda= {m}_\lambda$.

\begin{proposition}[\cite{Drensky}]
Let $\psi_n^m (A)$ be the $GL_m$-character of $A$ as in equation \eqref{eq:char_decomp}. Then $\widetilde{m}_\lambda \neq 0$ if and only if there exists a multitableau $T_\lambda$ such that $f_{T_\lambda} \notin \textnormal{Id} (A)$. Furthermore, $\widetilde{m}_\lambda$ is equal to the maximal number of linearly independent highest weight vectors $f_{T_\lambda}$ in $F_{m}^n(A)$.
\end{proposition}

In what follows, the next remark will be useful.

\begin{remark}\label{comparacaochar}
	Let $A$, $B$ and $A\oplus B$ be algebras with $n$-th cocharacter 
	$$ \chi_{{n}}(A)=\displaystyle \sum_{{\lambda}\vdash \gen{n}}{m}_{{\lambda}}\chi_{{\lambda}},\quad \chi_{{n}}(B)=\displaystyle \sum_{{\lambda}\vdash {n}}{m'}_{{\lambda}}\chi_{{\lambda}}\quad \mbox{ and } \quad \chi_{{n}}(A\oplus B)=\displaystyle \sum_{{\lambda}\vdash \gen{n}} \widetilde{m}_{{\lambda}}\chi_{{\lambda}},$$ respectively. Then, $\widetilde{m}_{{\lambda}} \leq {m}_{{\lambda}}+{m'}_{{\lambda}}$, for all  ${\lambda} \vdash {n}$. Moreover, if $B \in \textnormal{var}(A)$ then ${m'}_{{\lambda}} \leq {m}_{{\lambda}}$, for all ${\lambda} \vdash {n}$. As a consequence, we have $l_n(B) \leq l_n(A)$, for all $n$.
\end{remark}

In this paper, we adopt the notation established in \cite{Daniela} for the algebras discussed here. To remain consistent with that framework, we introduce new notation for some algebras previously presented and clarify the correspondence below.

We start by considering the following subalgebras of $UT_2:$
$$A_1=A_{2,1}= F{e_{11}}+Fe_{12}\mbox{ and }A_1^*=A_{2,1}^*= F{e_{22}}+Fe_{12}.$$

\begin{lemma} \cite{GLa} \label{colengths1} For the algebras $A_1$ and $A_1^*$ we have  $\textnormal{Id}(A_1)=\langle [x,y]z \rangle_T,\, \textnormal{Id}(A_1^*)=\langle x[y,z] \rangle_T$. Moreover,  $l_n(A_1)=l_n(A_1^*)=2$ and $l_n(A_1\oplus A_1^*)=3$, for $n>3$. 
\end{lemma}

Define the following subalgebras of $UT_3:$

\begin{align*}
A_2 \; &=\; N_3 = F(e_{11}+e_{22}+e_{33}) + Fe_{12} +Fe_{13}+ Fe_{23} , \\[4pt]
A_4 \; &=\; A_{3,1} = Fe_{11} + Fe_{12} + Fe_{13} + Fe_{23}, \\[4pt]
A_4^* \; &=\; A_{3,1}^* = Fe_{33} + Fe_{12} + Fe_{13} + Fe_{23}, \\[4pt]
A_5 \; &=\; Fe_{22} + Fe_{12} + Fe_{13} + Fe_{23}, \\[4pt]
A_6 \; &=\; F(e_{11}+e_{33}) + Fe_{12} + Fe_{13} + Fe_{23}. \\
\end{align*}

In \cite{Sandra}, the authors provided a finite generating set for the $\textnormal{T}$-ideal of $A_6$, listed below.

\begin{align*}
f_1 \; &=\; [x,y][z,w] + [z,w][x,y] + [y,z][x,w] + [x,w][y,z] 
          + [z,x][y,w] + [y,w][z,x], \\[6pt]
f_2 \; &=\; [x,y][z,w] + [z,w][x,y] + x[z,w]y - y[z,w]x, \\[6pt]
f_3 \; &=\; 2xy[w,z] + wx[z,y] + wy[z,x] + zx[y,w] + zy[x,w] \\ 
       &\quad + [x,y][z,w] + [x,z][w,y] - [w,y][x,z] 
              + [y,z][w,x] - [w,x][y,z], \\[6pt]
f_4 \; &=\; [x,y]z[w,t],\quad  f_5 \; =\; [[x,y][z,w],t] \quad \mbox{and}\quad f_6 \; =\; [z[x,y]w,t]. \\
\end{align*}

\begin{lemma} \label{tideala6} \cite{GLa, Gut,  Sandra} For the algebras $ A_2, \  A_4, \ A_4^*, \ A_5$ and $A_6$ we have  $\textnormal{Id}(A_2)=\langle [x,y,z],[x,y][z,w] \rangle_T,$
$\textnormal{Id}(A_4)=\langle [x,y]zw \rangle_T,\, \textnormal{Id}(A_4^*)=\langle xy[z,w] \rangle_T, \, \textnormal{Id}(A_5)=\langle x[y,z]w \rangle_T\, \mbox{ and } \, \textnormal{Id}(A_6)=\langle f_1,\ldots , f_6 \rangle_T.$
\end{lemma}

Furthermore, for the direct sum of the previous algebras, we have.

\begin{lemma} \cite{GLa, Daniela}
For the algebra $A_2$ we have    $l_n(A_2)=3$, $n> 3$. Moreover, $l_n(B_1)=4$ and $l_n(B_2)=5$, for all $n>3$, $B_1\in \{A_1\oplus A_2, A_1^*\oplus A_2\}$ and $B_2\in \{A_1\oplus A_1^*\oplus A_2, A_4,A_4^*,A_5,A_6\}$. 
\end{lemma}

Now, we consider the following subalgebras of $UT_4:$

\begin{align*}
A_7 \; &=\; F(e_{11}+e_{22}+e_{33}+e_{44}) 
          + Fe_{12} + Fe_{13} + Fe_{14} + Fe_{23} + Fe_{24} + Fe_{34}, \\[6pt]
A_8 \; &=\; F(e_{11}+e_{22}+e_{33}) 
          + Fe_{12} + Fe_{13} + Fe_{14} + Fe_{23} + Fe_{24} + Fe_{34}, \\[6pt]
A_8^* \; &=\; F(e_{22}+e_{33}+e_{44}) 
            + Fe_{12} + Fe_{13} + Fe_{14} + Fe_{23} + Fe_{24} + Fe_{34}, \\[6pt]
A_9 \; &=\; A_{4,1} 
          = Fe_{11} + F(e_{23}+e_{34}) + Fe_{12} + Fe_{13} + Fe_{14} + Fe_{24}, \\[6pt]
A_9^* \; &=\; A_{4,1}^* 
            = Fe_{44} + F(e_{12}+e_{23}) + Fe_{13} + Fe_{14} + Fe_{24} + Fe_{34}. \\ 
\end{align*}

\begin{remark}
Observe that, despite $N_4$ and $A_7$ being distinct algebras, they are PI-equivalent. Recall that, according to \cite[Theorems 5.5]{RN}, for the algebra  $N_4$ we have 
$$\chi_n(N_4)=\chi_{(n)}+2\chi_{(n-1,1)}+2\chi_{(n-2,1^2)}+\chi_{(n-2,2)}+\chi_{(n-3,2,1)}\mbox{ and } l_n(N_4)=7.$$ 
\end{remark}

\begin{lemma} \cite{Daniela, Danielasub, RN} \label{A7} For the algebras $A_7$, $A_8$, $A_8^*$, $A_9$ and $A_9^*$ we have
$$\textnormal{Id}(A_7)=\langle [x_1,x_2][x_3,x_4], [x_1,x_2,x_3,x_4] \rangle_T,$$ $$\textnormal{Id}(A_8)=\langle [x_1,x_2][x_3,x_4]x_5, [x_1,x_2,x_3]x_4 \rangle_T, \ \textnormal{Id}(A_8^*)=\langle x_5[x_1,x_2][x_3,x_4], x_4[x_1,x_2,x_3] \rangle_T,$$ $$\textnormal{Id}(A_9)=\langle [x_1,x_2][x_3,x_4], [x_1,x_2]x_3x_4x_5 \rangle_T \ \ \mbox{and} \ \ \textnormal{Id}(A_9^*)=\langle [x_1,x_2][x_3,x_4], x_1x_2x_3[x_4,x_5] \rangle_T.$$ \ 
 Moreover,
for all $n>4$ we have $l_n(A_7)=7$, $l_n(A_8)=l_n(A_8^*)=8$ and $l_n(A_9)=l_n(A_9^*)=10$.
\end{lemma}

Consider the following subalgebras of $UT_5:$

\begin{align*}
A_{10} \; &=\;Fe_{11}+Fe_{12}+Fe_{13}+Fe_{14}+Fe_{15}+F(e_{23}+e_{45})+F(e_{24}-e_{35})+Fe_{25},\\[6pt]
A_{10}^* \; &=\;Fe_{55}+Fe_{15}+Fe_{25}+Fe_{35}+Fe_{45}+F(e_{12}+e_{34})+F(e_{13}-e_{24})+Fe_{14}. \\ 
\end{align*}
\begin{lemma}
    For the algebras $A_{10}$ and $A_{10}^*$ we have:

\begin{enumerate}
    \item[1)] $\textnormal{Id}(A_{10}) = \langle [x_1, x_2]x_3x_4x_5\rangle_T$ and  $\textnormal{Id}(A_{10}^*) = \langle  x_1x_2x_3[x_4, x_5]\rangle_T$.
    \item[2)] $c_n(A_{10})=c_n(A_{10}^*)=n(n-1)(n-2).$
\end{enumerate}
    
\end{lemma}

\begin{proof}
   Note that 
\[
A_{10}\;\sim_{T}\;
\begin{pmatrix}
A_{4} & A_{4} \\
0 & 0
\end{pmatrix}
\]
and therefore, by \cite{Gut}, we have $$\textnormal{Id}(A_{10})=\langle \textnormal{Id}(A_{4})\, x \rangle_{T}
   = \langle  [x_{1},x_{2}]x_{3}x_{4}x_{5} \rangle_{T}\quad \mbox{ and }\quad \textnormal{Id}(A_{10}^*)
   = \langle  x_1x_2x_3[x_4, x_5] \rangle_{T},$$  and $
c_{n}(A_{10})=c_{n}(A_{10}^*) = n\, c_{n-1}(A_{4})=n(n-1)(n-2).$
\end{proof}
\begin{remark}
 Observe that, by the preceding lemma, we have  $A_{9} \in \mathrm{var}(A_{10})$. Therefore, $$l_{n}(A_{10}^*) = l_{n}(A_{10}) \geq l_{n}(A_{9})= 10,\mbox{ for all }n>4.$$
\end{remark}

\black


\begin{lemma}\label{G2k} \cite{Petro, RN}
 For the algebra $\mathcal{G}_{2k}$  we have $\textnormal{Id}(\mathcal{G}_{2k})=\langle \textnormal[x_1,x_2,x_3], [x_1,x_2][x_3,x_4]\cdots [x_{2k+1},x_{2k+2}] \rangle_{T}$, $\chi_n(\mathcal{G}_{2k})=\sum_{i=0}^{2k} \chi_{(n-i,1^i)}$ and $l_n(\mathcal{G}_{2k})=2k+1$, for all $k\geq 1$.
\end{lemma}

Consider the algebras below $$\overline{{\mathcal{G}_4}}=\begin{pmatrix}
\mathcal{G}_4 & \mathcal{G}_4 \\
0 & 0
\end{pmatrix}\cong \mathcal{G}_4\otimes A_1\quad \mbox{ and }\quad {\overline{\mathcal{G}_4^*}}=\begin{pmatrix}
0 & \mathcal{G}_4 \\
0 & \mathcal{G}_4
\end{pmatrix}\cong \mathcal{G}_4\otimes A_1^*.$$

\begin{lemma} \label{g4barra1} 
    For the algebra ${\overline{\mathcal{G}_4}}$ we have:
\begin{enumerate}
    \item[1)] $\textnormal{Id}({\overline{\mathcal{G}_4}})=\langle [x_1,x_2,x_3]x_4, [x_1,x_2][x_3,x_4][x_5,x_6]x_7\rangle_T.$ 
     
\item[2)] $c_n({\overline{\mathcal{G}_4}})=n+(n-2)\binom{n}{2}+(n-4)
\binom{n}{4}.$

\item[3)] The polynomials 
$$x_1\cdots x_n,\, x_{i_1}\cdots x_{i_{n-2}}[x_n,x_t],\, x_{i_1}\cdots x_{i_{n-3}}[x_i,x_j]x_k,\, x_{i_1}\cdots x_{i_{n-5}}[x_{l},x_m][x_p,x_q]x_r,$$ where $i_1< \cdots <i_{n-2},$ $i<j$ and $l<m<p<q$, constitute a basis of $P_n$ modulo $P_n\cap \textnormal{Id}({\overline{\mathcal{G}_4}}).$
\end{enumerate}

Similarly, $\textnormal{Id}({\overline{\mathcal{G}_4^*}})=\langle x_1[x_2,x_3,x_4], x_1[x_2,x_3][x_4,x_5][x_6,x_7]\rangle_T$.

\end{lemma}

\begin{proof}
    According to \cite{Gut}, $\textnormal{Id}({\overline{\mathcal{G}_4}})= 
    \langle \textnormal{Id}(\mathcal{G}_4)x \rangle_T= \langle [x_1,x_2,x_3]x_4, [x_1,x_2][x_3,x_4][x_5,x_6]x_7 \rangle_T  $ and $c_n({\overline{\mathcal{G}_4}})=nc_{n-1}(\mathcal{G}_4)=n+(n-2)\binom{n}{2}+(n-4)
\binom{n}{4}$. 

Now, we prove that the polynomials defined in $3)$ are linearly independent modulo $\textnormal{Id}({\overline{\mathcal{G}_4}})$. Consider 
$$f=\displaystyle \alpha x_1\cdots x_n+ \underset{\underset{i_1<\cdots <i_{n-2}}{t=1}}{\overset{{n-1}}{\sum}} \beta_t x_{i_1}\cdots x_{i_{n-2}}[x_n,x_t]+\underset{\underset{i_1<\cdots <i_{n-3} }{i<j,\,k}}{\sum} \gamma_{i,j,k} x_{i_1}\cdots x_{i_{n-3}}[x_i,x_j]x_k+$$ $$ \underset{\underset{i_1<\cdots <i_{n-5}}{l<m<p<q}}{\sum} \delta_{l,m}^{p,q,r} x_{i_1}\cdots x_{i_{n-5}}[x_{l},x_m][x_p,x_q]x_r\in \textnormal{Id}({\overline{\mathcal{G}_4^*}}).$$

First, considering the evaluation $x_v\mapsto 
1 \otimes e_{11}$, for all $v\in \{1,\ldots , n\}$, we have $\alpha=0$. Second, we fix $1\leq t\leq n-1$ and considering the evaluation $x_t \mapsto  1 \otimes e_{12}$ and $x_i\mapsto 
1 \otimes e_{11}$, for all $i\neq t$, we obtain $\beta_t=0$. Third, we fix $i,j,k$, with $i<j$, the evaluation $x_{i}\mapsto e_{1}\otimes e_{11},$ $x_{j}\mapsto e_2\otimes e_{11},$ $x_{k}\mapsto 1\otimes e_{12}$ and $x_{v}\mapsto 1\otimes e_{11},$ for all $v\notin \{i,j,k\}$, gives $\gamma_{i,j,k}=0$. Finally, we fix $l,m,p,q,r$, with $l<m<p<q$, the evaluation $x_{l}\mapsto e_{1}\otimes e_{11},$ $x_{m}\mapsto e_{2}\otimes e_{11},$ $x_{p}\mapsto e_{3}\otimes e_{11}$, $x_{q}\mapsto e_{4}\otimes e_{11}$ and $x_{r}\mapsto 1\otimes e_{12}$ and $x_{v}\mapsto {1}\otimes e_{11},$ for all $v\notin \{l,m,p,q,r\}$, gives $\delta_{l,m}^{p,q,r}=0$.

Since the number of polynomials defined in $3)$ is equal to the $n$-th codimension of ${\overline{\mathcal{G}_4}}$ and we have proved that they are linearly independent modulo $\textnormal{Id}({\overline{\mathcal{G}_4}})$, then we conclude that they constitute a basis of $P_n$ modulo $P_n\cap \textnormal{Id}({\overline{\mathcal{G}_4}})$.

\end{proof}

\begin{remark}
By \cite[Lemma 9]{Daniela}, we have $A_8\in \textnormal{var}({\overline{\mathcal{G}_4}})$. Therefore, $l_n({\overline{\mathcal{G}_4^*}})=l_n({\overline{\mathcal{G}_4}})\geq l_n(A_8)=8.$

\end{remark}

In the following, we consider the direct sum of some of the previous algebras. 

\begin{lemma} \label{idealg4a1}
    For the algebra $\mathcal{G}_4\oplus A_1$ and $\mathcal{G}_4\oplus A_1^*$ we have

    \begin{enumerate}
        \item[1)]  $\textnormal{Id}(\mathcal{G}_4\oplus A_1^*)=\langle x_1[x_2,x_3,x_4],  [x_1,x_2][x_3,x_4][x_5,x_6] \rangle_T.$

\item[2)] $\textnormal{Id}(\mathcal{G}_4\oplus A_1)=\langle [x_1,x_2,x_3]x_4,  [x_1,x_2][x_3,x_4][x_5,x_6] \rangle_T.$

\item[3)] $\chi_n(\mathcal{G}_4\oplus A_1^*)=\chi_n(\mathcal{G}_4\oplus A_1)= \chi_{(n)}+2\chi_{(n-1,1)}+\chi_{(n-2,1^2)}+\chi_{(n-3,1^3)}+\chi_{(n-4,1^4)}.$ 

\item[4)] $l_n(\mathcal{G}_4\oplus A_1^*)=l_n(\mathcal{G}_4\oplus A_1)=6$.
    \end{enumerate}
\end{lemma}

\begin{proof}

We start by proving the third item, and so consider $\chi_n(\mathcal{G}_4\oplus A_1^*)=\sum_{\lambda\vdash n} m_
\lambda \chi_{\lambda} $. According to Remark  \ref{comparacaochar} and \cite[Lemmas 3]{GLa} we have $$m_{\lambda}=1,
\quad \quad 1\leq m_{(n-1,1)}\leq 2\quad \mbox{ and } 
\quad m_{
\mu}=0,$$ for all $\lambda,\mu \vdash n$ with $\lambda\in \{(n), (n-2,1^2),(n-3,1^3),(n-4,1^4)\}$ and $\mu\notin \{\lambda, (n-1,1)\}$. Observe that if $m_{(n-1,1)} \neq 2$, then $\chi_n(\mathcal{G}_4\oplus A_1^*)=\chi_n(\mathcal{G}_4)$. Since $\mathcal{G}_4\in \textnormal{var}(\mathcal{G}_4\oplus A_1^*)$ then we must have $P_n\cap \textnormal{Id}(\mathcal{G}_4\oplus A_1^*)=  P_n\cap \textnormal{Id}(\mathcal{G}_4)$, a contradiction, since the polynomial $[x_1,x_2,x_3]x_4\cdots x_n\in P_n\cap \textnormal{Id}(\mathcal{G}_4)\setminus  P_n\cap \textnormal{Id}(\mathcal{G}_4\oplus A_1^*)$. Therefore, $m_{(n-1,1)}=2$ and the proof follows.

For the first item, consider $$I=\langle x_1[x_2,x_3,x_4],  [x_1,x_2][x_3,x_4][x_5,x_6]\rangle .$$ It is straightforward to check that $I\subseteq \textnormal{Id}(\mathcal{G}_4\oplus A_1^*)$.

 By the Poincaré-Birkhoff-Witt theorem, any element of $P_n$ can be written modulo $\textnormal{Id}(\mathcal{G}_4\oplus A_1^*)$ as a linear combination of the polynomials of the types

\begin{enumerate}
    \item[(1)] $x_{1}\cdots x_n,$
\item[(2)] $x_{i_1}\cdots x_{i_{n-2}}[x_{a},x_b]$, with $i_1<\cdots <i_{n-2}$ and $a<b$, 

\item[(3)] $x_{i_1}\cdots x_{i_{n-4}}[x_{c},x_d][x_e,x_f]$, with $i_1<\cdots <i_{n-4}$,

\item[(4)] $[x_i,x_j,x_k]x_{i_1}\cdots x_{i_{n-3}}$, with $i_1<\cdots <i_{n-3}$. 

\end{enumerate}

Note that the polynomial $w[y,x][y,z]$ is a consequence of $x_1[x_2,x_3,x_4]$. By linearization, we have $w[y_1,x][y_2,z]\equiv -w[y_2,x][y_1,z]$ modulo $\textnormal{Id}(\mathcal{G}_4\oplus A_1^*)$ and so we can assume $c<d<e<f$ in the polynomials of type $(3)$.    

For polynomials of type $(4)$, we observe that the polynomial $h=[[x_1,x_2],[x_3,x_4]]$ is a consequence of $f=x_1[x_2,x_3,x_4]$. Now, we note that the polynomial $p=[x_1,x_2,x_3]x_4-[x_1,x_2,x_4]x_3$ is a consequence of $h$ and $f$. The polynomial $p$ with the identity $[x_1,x_2,x_3]=[x_1,x_3,x_2]+[x_3,x_2,x_1]$, allows us to assume $i=1$ and $k<i_1<\cdots <i_{n-3}$.

Since $c_n(\mathcal{G}_4\oplus A_1^*))=
\chi(\mathcal{G}_4\oplus A_1^*))(1)=1+\binom{n}{2}+\binom{n}{4}+(n-1)$ and the previous polynomials generate $P_n$ modulo $\textnormal{Id}(\mathcal{G}_4\oplus A_1^*))$ then they must be linearly independent modulo  $\textnormal{Id}(\mathcal{G}_4\oplus A_1^*))$. Therefore, we conclude that $I=\textnormal{Id}(\mathcal{G}_4\oplus A_1^*)$.

The second item is proved similarly.

\end{proof}

An argument analogous to the one used in the proof of the previous lemma 
can be applied to the algebras $A_4 \oplus A_1$, $A_4 \oplus A_1^*$, $A_4^* \oplus A_1$, $A_4^* \oplus A_1^*$, $\mathcal{G}_4 \oplus A_1 \oplus A_1^*$, $\mathcal{G}_4 \oplus A_4$ and $\mathcal{G}_4 \oplus A_4^*$. For the remaining direct sums of algebras, it is possible to construct a sufficient number of linearly independent highest weight vector in order to compute the corresponding multiplicities. Hence, we omit the proof.

\begin{lemma} \label{charac}
 For the direct sum of the previous algebras, we have: 
\begin{enumerate}
    \item[1)] $\chi_n(A_4\oplus A_1^*)=\chi_n(A_4^*\oplus A_1)= \chi_{(n)}+3\chi_{(n-1,1)}+\chi_{(n-2,1^2)}+\chi_{(n-2,2)}$ and $l_n(A_4\oplus A_1)=l_n(A_4\oplus A_1^*)=6$.

       \item[2)] $\chi_n(B\oplus A_2)=\chi_{(n)}+3\chi_{(n-1,1)}+2\chi_{(n-2,1^2)}+\chi_{(n-2,2)}$ and $l_n(B\oplus A_2)=7$, for all $B \in \{A_4, A_4^*,A_5, A_6\}.$

          \item[3)] $\chi_n(B_1\oplus B_2)=\chi_{(n)}+4\chi_{(n-1,1)}+2\chi_{(n-2,1^2)}+2\chi_{(n-2,2)}$ and $l_n(B_1\oplus B_2)=9$, for all $B_1\neq B_2$ and $B_1,B_2\in \{A_4,A_4^*, A_5, A_6\}$.

    \item[4)] $\chi_n(\mathcal{G}_4\oplus A_1\oplus A_1^*)= \chi_{(n)}+3\chi_{(n-1,1)}+\chi_{(n-2,1^2)}+\chi_{(n-3,1^3)}+\chi_{(n-4,1^4)}$  and $l_n(\mathcal{G}_4\oplus A_1\oplus A_1^*)=7$.

    \item[5)] $\chi_n(\mathcal{G}_4\oplus A_4)= \chi_n(\mathcal{G}_4\oplus A_4^*)= \chi_{(n)}+3\chi_{(n-1,1)}+2\chi_{(n-2,1^2)}+\chi_{(n-2,2)}+ \chi_{(n-3,1^3)}+\chi_{(n-4,1^4)}$ and $l_n(\mathcal{G}_4\oplus A_4)=l_n(\mathcal{G}_4\oplus A_4^*)=9$.
\end{enumerate}
\end{lemma}


 \section{\texorpdfstring{Classifying varieties of bounded colength}{Classifying varieties of bounded colength}}

In this section, we classify the varieties whose colength sequence is bounded by 7. We note that the results presented here extend the existing literature, as previous progress in this direction has only established classifications for varieties with colength sequences bounded above by 4.

According to Theorem \ref{teokemer}, the algebras of the type $F+J(A)$ play an important role in the classification of varieties with bounded colength. Recall that if $A=F+J(A)$ is a finite-dimensional algebra then, by \cite[Lemma 9.7.1]{GiambrunoZai},  we can decompose \begin{equation} \label{decompj}
    J(A)=J_{10}+J_{01}+J_{00}+J_{11}, \mbox{ where }
\end{equation}    $$J_{00}=\{j\in J\mid 1_Fj=j1_F=0\},\quad J_{10}=\{j\in J\mid 1_Fj=j 
    \mbox{ and }j1_F=0\}$$
      $$J_{01}=\{j\in J\mid 1_Fj=0 
    \mbox{ and }j1_F=j\}\quad \mbox{ and }\quad J_{11}=\{j\in J\mid 1_Fj=j1_F=j\}.$$ Furthermore, for $i, k,r,s \in \{0, 1\}$, $J_{ik}J_{rs} \subseteq \delta_{kr} J_{is}$, where $\delta_{kr}$ is the Kronecker delta function.
    
We begin by recalling several lemmas from \cite{GLa} and \cite{Daniela}, which we gather in the next two lemmas.

\begin{lemma} \cite[Lemmas 9,10, 15, 20, 21]{GLa} \label{1} Let $A=F+J(A)$ be an $F$-algebra.
\begin{enumerate}
 \item[1)]If $J_{10}\neq \{0\}$ then $A_{1}\in \textnormal{var}(A)$.

  \item[2)]If $J_{01}\neq \{0\}$ then $A_{1}^*\in \textnormal{var}(A)$.

   \item[3)]If $A_{2}\notin \textnormal{var}(A)$ then $[J_{11}, J_{11}] = \{0\}$.

   \item[4)] If $A_4\notin \textnormal{var}(A)$ then $J_{10}J_{00}= \{0\}$.
    \item[5)] If $A_4^*\notin \textnormal{var}(A)$ then $J_{00}J_{01}=\{0\}$.
\item[6)]  If $J_{10}J_{00}=J_{00}J_{01}= \{0\}$ and $J_{01}J_{10}\neq \{0\}$ then $A_5\in \textnormal{var}(A).$

\item[7)] If $J_{10}J_{01}\neq \{0\}$ and $J_{01}J_{10}=\{0\}$ then $A_6\in \textnormal{var}(A).$

\end{enumerate}
\end{lemma}







\begin{lemma} \cite[Lemmas 13 and 16]{Daniela} \label{2}
    Let $A=F+J(A)$ be an $F$-algebra. 
    
    \begin{enumerate}
        \item[1)]     If $A_7\notin \textnormal{var}(A)$ then $[J_{11},J_{11},J_{11}]=\{0\}$.

        \item[2)] If $\mathcal{G}_4, A_7, A_8,A_8^*\notin \textnormal{var}(A)$ and $J_{00}=J_{10}J_{01}=\{0\}$ then $A\in \textnormal{var}(UT_2).$
    \end{enumerate}

\end{lemma}

Now, we prove a sequence of auxiliary results that will support the main result.

\begin{lemma}  \label{g4barra}
    Let $A=F+J(A)$ be an algebra such that $[J_{11},J_{11},J_{11}]=[J_{11},J_{11}][J_{11},J_{11}][J_{11},J_{11}]=\{0\}$. 

\begin{enumerate}
    \item[1)] If $[J_{11},J_{11}][J_{11},J_{11}]J_{10}\neq \{0\}$ then ${\overline{\mathcal{G}_4}}\in \textnormal{var}(A).$ 
     \item[2)] If $J_{01}[J_{11},J_{11}][J_{11},J_{11}]\neq \{0\}$ then ${\overline{\mathcal{G}_4^*}}\in \textnormal{var}(A).$ 
\end{enumerate}
\end{lemma}

\begin{proof}

Since the second item is proved similarly, we focus on the first item.

    Let $A=F+J_{10}+J_{01}+J_{11}+J_{00}$ as given in (\ref{decompj}) and consider $B=F+J_{10}+J_{11}$ the subalgebra of $A$. Since \( [J_{11}, J_{11}, J_{11}] = \{0\} \), we have \( [B, B, B]\subseteq J_{10} \). Therefore, $[x_1,x_2,x_3]x_4\equiv 0$ on $B$.

Moreover, since
\[
[J_{11}, J_{11}] [J_{11}, J_{11}] [J_{11}, J_{11}] = [F, J_{11}] = [F, F] = \{0\},
\] we also have \( [B, B][B, B][B, B]\subseteq J_{10} \) and so $[x_1, x_2][x_3, x_4][x_5, x_6]x_7 \equiv 0$ on $B$. Therefore, by Lemma \ref{g4barra1}, $\textnormal{Id}({\overline{\mathcal{G}_4}})\subseteq \textnormal{Id}(B).$

In order to prove the opposite inclusion, let $f\in P_n\cap \textnormal{Id}(B)$. According to Lemma \ref{g4barra1}, we can write $f$ modulo $\textnormal{Id}({\overline{\mathcal{G}_4}})$ as $$\displaystyle \alpha x_1\cdots x_n+ \underset{\underset{i_1<\cdots <i_{n-2}}{t=1}}{\overset{{n-1}}{\sum}} \beta_t x_{i_1}\cdots x_{i_{n-2}}[x_n,x_t]+\underset{\underset{i_1<\cdots <i_{n-3} }{i<j,\,k}}{\sum} \gamma_{i,j,k} x_{i_1}\cdots x_{i_{n-3}}[x_i,x_j]x_k+$$ $$ \underset{\underset{i_1<\cdots <i_{n-5}}{l<m<p<q}}{\sum} \delta_{l,m}^{p,q,r} x_{i_1}\cdots x_{i_{n-5}}[x_{l},x_m][x_p,x_q]x_r.$$

Since $[J_{11},J_{11}][J_{11},J_{11}]J_{10}\neq \{0\}$, there exist $a,b,c,d\in J_{11}$ and $e\in J_{10}$ such that $[a,b][c,d]e\neq 0$. First, considering the evaluation $x_v\mapsto 1_F$, for all $v\in \{1,\ldots , n\}$, we have $\alpha=0$. Second, we fix $1\leq t\leq n-1$ and considering the evaluation $x_t \mapsto e$ and $x_i\mapsto 1_F$, for all $i\neq t$, we obtain $\beta_t=0$. Third, we fix $i,j,k$, with $i<j$, the evaluation $x_{i}\mapsto c,$ $x_{j}\mapsto d,$ $x_{k}\mapsto e$ and $x_{t}\mapsto 1_F,$ for all $t\notin \{i,j,k\}$, gives $\gamma_{i,j,k}=0$. Finally, we fix $l,m,p,q,r$, with $l<m<p<q$, the evaluation $x_{l}\mapsto a,$ $x_{m}\mapsto b,$ $x_{p}\mapsto c$, $x_{q}\mapsto d$ and $x_{r}\mapsto a$ and $x_{v}\mapsto 1_F,$ for all $v\notin \{l,m,p,q,r\}$, gives $\delta_{l,m}^{p,q,r}=0$.

Therefore, we have proved that $f\in \textnormal{Id}({\overline{\mathcal{G}_4}})$ and then $ {\overline{\mathcal{G}_4}}\sim_T B \in \textnormal{var}(A).$

\end{proof}

Consider $$\mathcal{I}_1=\{\mathcal{G}_4\oplus A_1\oplus A_1^*,\mathcal{G}_4\oplus A_4, \mathcal{G}_4\oplus A_4^*, A_7, \mathcal{G}_{6}, {\overline{\mathcal{G}_4}}, {\overline{\mathcal{G}_4^*}} \}.$$

\begin{lemma} \label{lemmag4}  Let $A=F+J(A)$ such that $\mathcal{G}_4\in \textnormal{var}(A)$. If $Q\notin \textnormal{var}(A)$, for all $Q\in \mathcal{I}_1$, then $A\sim_T B
\oplus N$, where $B\in \{\mathcal{G}_4,\mathcal{G}_4 \oplus A_1, \mathcal{G}_4 \oplus A_1^* \}$ and $N$ is a nilpotent algebra.
\end{lemma}

\begin{proof}
    As long as \( A_7 \notin \textnormal{var}(A) \), by Lemma~\ref{2}(1), we have $[J_{11}, J_{11}, J_{11}] = \{0\}$. Therefore, \( F + J_{11} \in \textnormal{var}(\mathcal{G}) \). Since \( \mathcal{G}_6 \notin \textnormal{var}(A) \), by Theorem~\ref{subvarieties}, it follows that $F + J_{11} \sim_T B_1 \oplus N,$ where \( B_1 \in \{F, A_2, \mathcal{G}_4\} \) and \( N \) is a nilpotent algebra. In all such cases, we obtain
\[
[J_{11}, J_{11}] [J_{11}, J_{11}] [J_{11}, J_{11}] = \{0\}.
\]

Now, assume that \( J_{10} = J_{01} = \{0\} \). Then $A = (F + J_{11}) \oplus J_{00},$ and since \( \mathcal{G}_4 \in \textnormal{var}(A) \), it follows that
\[
A \sim_T \mathcal{G}_4 \oplus N.
\]

Suppose now that \( J_{10} \neq \{0\} \). According to Lemma~\ref{1}(1), we have \( A_1 \in \textnormal{var}(A) \), and therefore \( \mathcal{G}_4 \oplus A_1 \in \textnormal{var}(A) \). Since $\mathcal{G}_4 \oplus A_1 \oplus A_1^* \notin \textnormal{var}(A),$ we conclude that \( A_1^* \notin \textnormal{var}(A) \), and hence Lemma~\ref{1}(2) implies \( J_{01} = \{0\} \). Furthermore, as \( \mathcal{G}_4 \oplus A_4 \notin \textnormal{var}(A) \), we obtain \( A_4 \notin \textnormal{var}(A) \), and once again Lemma~\ref{1}(4) yields
$J_{10} J_{00} = \{0\}.$ Thus, we can write \( A = B \oplus J_{00} \), where \( B = F + J_{10} + J_{11} \).

Let us fix a basis of \( B \) given by the union of bases of \( F \), \( J_{10} \) and \( J_{11} \). Since \( [J_{11}, J_{11}, J_{11}] = \{0\} \), any nonzero evaluation of the polynomial \( [x_1, x_2, x_3] \) on basis elements of \( B \) necessarily requires at least one variable to be evaluated in $J_{10}$. Therefore, $[B,B,B]\subseteq J_{10}$ and so $[x_1,x_2,x_3]x_4\equiv 0$ on $B$.

Moreover, since
\[
[J_{11}, J_{11}] [J_{11}, J_{11}] [J_{11}, J_{11}] = [F, J_{11}] = [F, F] = \{0\},
\]
a nonzero evaluation of the polynomial \( [x_1, x_2][x_3, x_4][x_5, x_6] \) on the basis of \( B \) must occur such that \( x_1, x_2, x_3, x_4 \) are evaluated in \( J_{11} \), and among the variables \( x_5 \) and \( x_6 \), one is evaluated in \( F + J_{11} \) and the other in \( J_{10} \). In this case, we have $[B,B][B,B][B,B]\subseteq [J_{11},J_{11}][J_{11},J_{11}]J_{10}$. Since \( {\overline{\mathcal{G}_4}} \notin \textnormal{var}(A) \), by Lemma~\ref{g4barra}, we have
\[
[J_{11}, J_{11}] [J_{11}, J_{11}] J_{10} = \{0\}.
\]
Therefore, $[x_1, x_2][x_3, x_4][x_5, x_6] \equiv 0$ on $B$.

By Lemma~\ref{idealg4a1}, we conclude that \( B \in \textnormal{var}(\mathcal{G}_4 \oplus A_1) \) and hence $B \sim_T \mathcal{G}_4 \oplus A_1.$ Since \( A \sim_T B \oplus J_{00} \), the result follows.

Similarly, if $J_{01}\neq \{0\}$, a similar argument proves that, if $\mathcal{G}_4\oplus A_1\oplus A_1^*,\mathcal{G}_4\oplus A_4^*, {\overline{\mathcal{G}_4^*}}\notin \textnormal{var}(A)$ then $A\sim_T B\oplus N$, where $B=F+J_{01}+J_{11}\sim_T \mathcal{G}_4\oplus A_1^*.$

\end{proof}

\begin{lemma} \label{Lemma:A9}
    Let $A=F+J(A)$ be an algebra.
    
    \begin{enumerate}
        \item[1)] If $A_9\notin \textnormal{var}(A)$ then $ab^2= 0$, for all $a\in J_{10}$ and $b\in J_{00}$.

        \item[2)] If $A_9^*\notin \textnormal{var}(A)$ then $b^2a= 0$, for all $a\in J_{01}$ and $b\in J_{00}$.
    \end{enumerate}
\end{lemma}

\begin{proof}
  Assume that there exist $a \in J_{10}$ and $b \in J_{00}$ such that $ab^{2} \neq 0$. 
Let $R$ be the subalgebra of $A$ generated by $1,\, a$ and $b$, and let $I$ denote the $\textnormal{T}$-ideal of $R$ generated by $b^{3}$. 
It is straightforward to verify that $I$ is linearly spanned by the elements $b^{k}$ and $ab^{k}$ for all $k \geq 3$.

We claim that $ab^{2} \notin I$. 
Indeed, suppose by contradiction that $ab^{2} \in I$. 
Then we can write 
\[
    ab^{2} = \sum_{k \geq 3} \alpha_{k} ab^{k}, \qquad \alpha_k \in F.
\]
Let $i$ be the smallest positive integer such that $ab^{i} \neq 0$. 
Multiplying the above equality on the right by $b^{i-2}$ we obtain 
\[
    ab^{i} = \sum_{k \geq 3} \alpha_{k} ab^{k+i-2} = 0,
\]
which contradicts the minimality of $i$. 
Hence $ab^{2} \notin I$. Consequently, we have $a, b, ab, b^{2} \notin I$. 
Therefore, the quotient algebra $R/I$ is linearly generated by the elements $ \overline{1}, \ \overline{a}, \ \overline{b}, \ \overline{ab}, \ \overline{ab^{2}}$ and $ \overline{b^{2}}.$ It is straightforward to check that the map $\varphi \colon R/I \longrightarrow A_{9}$ given by $$
    \overline{1} \mapsto e_{11}, \quad 
    \overline{a} \mapsto e_{12}, \quad 
    \overline{b} \mapsto e_{23} + e_{34}, \quad  \overline{b^2} \mapsto e_{24}, \quad 
    \overline{ab} \mapsto e_{13}, \quad
    \overline{ab^{2}} \mapsto e_{14}$$ is an isomorphism. 
Thus, $A_{9} \cong R/I \in \mathrm{var}(R) \subseteq \mathrm{var}(A)$.

\end{proof}

\begin{lemma} \label{j10j00}
     Let $A=F+J(A)$ be an algebra.      
     \begin{enumerate}
         \item[1)]  If $A_9, A_{10}\notin \textnormal{var}(A)$ then $J_{10}J_{00}^2= \{0\}$. 

         \item[2)]  If $A_9^*, A_{10}^*\notin \textnormal{var}(A)$ then $J_{00}^2J_{01}= \{0\}$. 
     \end{enumerate}

\end{lemma}

\begin{proof}
    Suppose that there exist $a\in J_{10}$ and $b,c\in J_{00}$ such that $abc\neq 0$. Since $A_9\notin \textnormal{var}(A)$, by Lemma \ref{Lemma:A9}, we have $ab^2=ac^2=a(b+c)^2=0$ and so $abc+acb=0$. Let $R$ be the subalgebra of $A$ generated by $1,a$ and $b$ and consider $I$ the $\textnormal{T}$-ideal of $R$ generated by $b^2,c^2$ and $bc+cb$.

Since $I\subseteq J_{00}$ then $a, abc,\, ab, ac \notin I$. Therefore, we also have $b,c, bc, cb\notin I$. Now, we observe that the quotient $R/I$ is linearly generated by $$\overline{1},\, \overline{a},\, \overline{b},\, \overline{c},\,\overline{ab},\,\overline{ac},\,\overline{abc},\, \overline{bc},\, \overline{cb}$$ satisfying $\overline{b}^2=\overline{c}^2=\overline{ab^2}=\overline{ac^2}=\overline{bc}+\overline{cb}=\overline{abc}+\overline{acb}=0$. Therefore, $R/I\cong A_{10}$ through the isomorphism given by $$\overline{1}\mapsto e_{11},\quad  \overline{a}\mapsto e_{12},\quad \overline{b}\mapsto e_{23}-e_{45}\quad \mbox{and}\quad  \overline{c}\mapsto e_{24}+e_{35}.$$

\end{proof}
\black

Consider the following sets
$$\mathcal{I}_2=\{A_4\oplus A_2, A_4\oplus A_4^*, A_4\oplus A_5, A_4\oplus A_6, A_9,A_{10}\},$$
$$\mathcal{I}_3=\{A_4^*\oplus A_2,A_4^*\oplus A_4,A_4^*\oplus A_5, A_4^*\oplus A_6, A_9^*,A_{10}^*\}.$$

\begin{lemma} \label{lemmaa4}
    Let $A=F+J$ be an algebra. 

\begin{enumerate}
    \item[1)] If $J_{10}J_{00}\neq \{0\}$ and $Q\notin \textnormal{var}(A)$, for all $Q\in \mathcal{I}_2$, then either $A\sim_T A_4\oplus N$ or $A\sim_T A_4\oplus A_1^*\oplus N$.

    \item[2)] If $J_{00}J_{01}\neq \{0\}$ and $Q\notin \textnormal{var}(A)$, for all $Q\in \mathcal{I}_3$, then either $A\sim_T A_4^*\oplus N$ or $A\sim_T  A_4^*\oplus A_1\oplus N$.
\end{enumerate}    
\end{lemma}

\begin{proof}
    Since the proof of the second item follows similarly, we focus on the first item.

 Assume that $J_{10}J_{00}\neq \{0\}$ and $J_{01}\neq \{0\}$. By Lemma \ref{1}(2)(4), we have $A_1^*,A_4\in \textnormal{var}(A)$ and so $ \textnormal{Id}(A)\subseteq \textnormal{Id}(A_4)\cap \textnormal{Id}(A_1^*) $. Moreover, since $J_{00}\subseteq A$, it is immediately that $\textnormal{Id}(A)\subseteq \textnormal{Id}(J_{00})$. Therefore, $\textnormal{Id}(A)\subseteq \textnormal{Id}(A_4)\cap \textnormal{Id}(A_1^*)\cap \textnormal{Id}(J_{00}).$ We assert that $A\sim_T A_4\oplus A_1^*\oplus J_{00}$ in this case.

 Let $f\in P_n\cap \textnormal{Id}(A_4)\cap \textnormal{Id}(A_1^*)\cap \textnormal{Id}(J_{00})$. As long as $A_4\oplus A_2, A_4\oplus A_4^*\notin \textnormal{var}(A)$ then $A_2,A_4^*\notin \textnormal{var}(A)$. Moreover, since $A_9,A_{10}\notin \textnormal{var}(A)$, by Lemmas \ref{1}(3)(5) and \ref{j10j00}, we have 
 \begin{equation} \label{relacoes}
[J_{11},J_{11}]=J_{00}J_{01}=J_{10}J_{00}^2=0.
 \end{equation}

Suppose that $J_{01}J_{10}\neq \{0\}$ and consider $I=J_{10}J_{00}+J_{01}J_{10}J_{00}$. Observe that $I$ is in fact an ideal of $A$. Moreover, if $J_{01}J_{10}\subset I$ then $J_{01}J_{10}\subset J_{01}J_{10}J_{00}$, but this means that $J_{01}J_{10}J_{00}\subset J_{01}J_{10}J_{00}^2=\{0\}$. Therefore, we have $J_{01}J_{10}\not\subset I$ and so $R=A/I$ is an algebra satisfying 
$$J(R)_{10}J(R)_{00}=J(R)_{00}J(R)_{01}=\{0\}\mbox{ and } J(R)_{01}J(R)_{10}\neq \{0\}.$$ By Lemma \ref{1}(6), we have $A_5\in \textnormal{var}(R)\subset \textnormal{var}(A)$ and so $A_4\oplus A_5\in \textnormal{var}(A)$, a contradiction. Therefore, $J_{01}J_{10}= \{0\}$. 

Now, since $A_4\oplus A_6\notin \textnormal{var}(A)$ we have $A_6\notin \textnormal{var}(A)$. Thus, by Lemma \ref{1}(7) we have $J_{10}J_{01}=\{0\}$.

Consider a basis of $A$ consisting of the union of the bases of $F,J_{10},J_{11}$ and $J_{00}$. Therefore, a nonzero evaluation of $f$ on the basis of $A$ must be taken either on elements from $B_1=F+J_{11}+J_{01}$ or on $B_2=F+J_{11}+J_{10}+J_{00}$. In the first case, since $f\in \textnormal{Id}(A_1^*)$ then $f$ is a consequence of $f_1=x_1[x_2,x_3]$. Since $f_1\equiv 0$ on $B_1$ we have $f\equiv 0$ on $A$.

In the second case, we note that, by (\ref{relacoes}), a nonzero evaluation of $f$ on $B_2$ must be taken either on elements from $J_{00}$ or at most one variable is evaluated on $J_{00}$. In the first case, since $f\in \textnormal{Id}(J_{00})$ we get $f\equiv 0$ on $A$. In the second case, we notice that $f_2=[x_1,x_2]x_3x_4$ vanishes under all evaluation of $A$ in which at most one variable takes value on $J_{00}$. Since $f\in \textnormal{Id}(A_4)$, it follows from Lemma \ref{tideala6} that $f$ is a consequence of $f_2$ and then $f$ is an identity of $A$. 

Therefore, we obtain $\textnormal{Id}(A_4)\cap \textnormal{Id}(A_1^*)\cap  \textnormal{Id}(J_{00})\subseteq \textnormal{Id}(A)$ and so $A\sim_T A_4\oplus A_1^*\oplus N$, where $N=J_{00}$ is a nilpotent algebra.

Finally, the case when $J_{01} =\{0\}$ can be treated analogously by repeating the arguments above. In this case, we obtain $A\sim_T A_4\oplus N$.    
\end{proof}

In the following, let $$\mathcal{I}_4=\{A_5\oplus A_2, A_5\oplus A_6\}.$$

\begin{lemma} \label{lemmaa5}
    Let $A=F+J(A)$ be an algebra such that $J_{10}J_{00}=J_{00}J_{01}=\{0\}$. If $J_{01}J_{10}\neq \{0\}$ and $Q\notin \textnormal{var}(A)$, for all $Q\in \mathcal{I}_4$, then $A\sim_{T} A_5\oplus N.$ 
\end{lemma}

\begin{proof} 
Suppose that \(J_{01}J_{10} \neq \{0\}\). By Lemma~\ref{1}(6), we have \(A_5 \in \mathrm{var}(A)\). Consequently,
\[
\mathrm{Id}(A) \subseteq \mathrm{Id}(A_5)\cap \mathrm{Id}(J_{00}).
\]

Let \(f\in P_n\cap \mathrm{Id}(A_5)\cap \mathrm{Id}(J_{00})\) and fix a basis of \(A\) obtained by the union of bases of \(F\), \(J_{10}\), \(J_{01}\), \(J_{11}\) and \(J_{00}\). To show that \(f\equiv 0\) on \(A\), we first establish that \(J_{10}J_{01}=\{0\}\).

Assume, by contradiction, that \(J_{10}J_{01}\neq \{0\}\). By hypothesis, \(J_{01}J_{10}\) is a two-sided ideal of \(A\). Therefore, in the quotient algebra \(R = A/J_{01}J_{10}\) we have
\[
J(R)_{10}J(R)_{00}
= J(R)_{00}J(R)_{01}
= J(R)_{01}J(R)_{10}
= \{0\}
\quad \mbox{ and }\quad 
J(R)_{10}J(R)_{01}\neq \{0\}.
\]
By Lemma~\ref{1}(7), it follows that \(A_6\in \mathrm{var}(R)\subseteq \mathrm{var}(A)\). Since \(A_5\oplus A_6\notin \mathrm{var}(A)\), we obtain a contradiction. Hence, $J_{10}J_{01}=\{0\}.$

Next, note that \(A_5\oplus A_2\notin \mathrm{var}(A)\). Thus $A_2\notin \mathrm{var}(A)$ and, by Lemma~\ref{1}(3), we deduce that \([J_{11},J_{11}] = \{0\}\). Summarizing, we obtain
\[
J_{00}J_{01}
= J_{10}J_{00}
= J_{10}J_{01}
= [J_{11},J_{11}]
= \{0\}.
\]

These relations imply that any non-zero evaluation of \(f\) on basis elements of \(A\) must lie either in \(J_{00}\) or in the subspace $B = F + J_{10} + J_{01} + J_{11}.$  In the first case, since \(f\in \mathrm{Id}(J_{00})\) then  \(f\equiv 0\) on \(A\). If the evaluation is taking on elements of \(B\), then we note that the polynomial \(f_1 = x_1[x_2,x_3]x_4\) satisfies \(f_1\equiv 0\) on \(B\). Since \(f\in \mathrm{Id}(A_5)\), Lemma~\ref{tideala6} implies that \(f\) is a consequence of \(f_1\), and thus \(f\equiv 0\) on \(A\).

Therefore, we conclude that \( A \sim_T A_5\oplus J_{00} \), which completes the proof.
\black 

\end{proof}




\black

\begin{lemma} \label{lemmaa6}
    Let $A=F+J(A)$ be an algebra satisfying $J_{10}J_{00}=J_{00}J_{01}=J_{01}J_{10}=\{0\}$. If $J_{10}J_{01}\neq \{0\}$ and $A_6\oplus A_2\notin \textnormal{var}(A)$ then $A\sim_{T} A_6\oplus N.$ 
\end{lemma}

\begin{proof}
Similarly to the previous lemma, we notice that, under the hypothesis, we have $A = B \oplus J_{00}$, where $B=F + J_{10} + J_{01} + J_{11}$ is a subalgebra of $A$. Moreover, since $J_{01}J_{10}=\{0\}$ and $J_{10}J_{01}\neq \{0\}$, by Lemma \ref{1}(7), we obtain $A_6\in \textnormal{var}(B)$. Since $A_6\oplus A_2\notin \textnormal{var}(B)$ we have $A_2\notin \textnormal{var}(B)$ and so $[J_{11},J_{11}]=\{0\}$. Therefore, we have the identities $$[J_{11},J_{11}]=[J_{10}J_{01},F]=[J_{10}J_{01},J_{11}]=J_{01}J_{10}=\{0\}.$$ 
Now, by means of long and detailed computations, one can verify that, using the relations above, the subalgebra \( B \) satisfies the polynomials \( f_1, f_2, f_3, f_4, f_5 \) and \( f_6 \) from Lemma \ref{tideala6}. Hence, \( B \in \textnormal{var}(A_6) \), and consequently \( B \sim_T A_6 \).

Therefore, we have proved that \( A \sim_T A_6 \oplus N \).
\end{proof}

For convenience of notation, we set $\mathcal{I}_5=\{A_6\oplus A_2\}.$

\begin{remark} \label{remarksubv} Observe that $\mathcal{G}_6\in \textnormal{var}(\mathcal{G})$ and $A_4,A_4^*,N_4,A_9,A_9^*\in \textnormal{var}(UT_2).$
\end{remark}

We are now ready to prove the main result of this section. To this end, we define the following sets:
$$\mathcal{I}_6=\{ A_7,A_8,A_8^*\},\quad \mathcal{I}_7=\{A_6\oplus A_2, N_4, A_9,A_9^*, A_2\oplus A_4, A_2\oplus A_4^*, A_4\oplus A_4^* \} \quad \mbox{ and }\quad \mathcal{I}=\bigcup_{i=1}^7 \mathcal{I}_i.$$

\begin{theorem} \label{maintheorem}
		Let $A$ be an algebra over $F$. The following conditions are equivalent:
		\begin{enumerate}
			\item[1)]$Q\notin \textnormal{var}(A)$, for all $Q\in \mathcal{I} $.
			
			\item[2)] $A$ is PI-equivalent to one of the following algebras: $$N, C\oplus N, A_1\oplus N, A_1^*\oplus N, A_1\oplus A_1^*\oplus N, A_2\oplus N,$$
   $$  A_1\oplus A_2\oplus N,  A_1^*\oplus A_2\oplus N,  A_4\oplus N, A_4^*\oplus N, A_4\oplus A_1^*\oplus  N,$$
   $$  A_4^*\oplus A_1\oplus N,  A_5\oplus N, A_6\oplus N, \mathcal{G}_4\oplus N, \mathcal{G}_4\oplus A_1\oplus  N, \mathcal{G}_4\oplus A_1^* \oplus N,$$ where $N$ denotes a nilpotent algebra and $C$ is a commutative non-nilpotent algebra.

   \item[3)] $l_{n}(A)\leq 6,$ for $n$ large enough.
		\end{enumerate}
	\end{theorem}

\begin{proof} Since, for \(n\) sufficiently large, we have \(l_n(S)\geq 7\) for every \(S\in \mathcal{I}\), then condition \(3)\) implies condition \(1)\). Moreover, by the results established in Section~3, we have already shown that \(2)\) implies \(3)\).

Assume now that $Q \notin \textnormal{var}(A)$ for all $Q \in \mathcal{I}$. By Remark \ref{remarksubv}, we have $\mathcal{G}, UT_2 \notin \textnormal{var}(A)$, and hence, by Theorem \ref{teokemer}, we obtain
$$
A \sim_{T} B_1 \oplus \cdots \oplus B_m,
$$
where each $B_i$ is a finite-dimensional algebra such that either $B_i$ is nilpotent or $B_i \cong F + J(B_i)$. 

If $B_i$ is nilpotent for every $i$, then $A \sim_T N$, where $N$ denotes a nilpotent algebra. Thus, we may assume that $B_i \cong F + J(B_i)$, for some $i \in \{1, \ldots, m\}$.

Let us write $J(B_i) = J_{10} + J_{01} + J_{00} + J_{11}$ the decomposition of the Jacobson radical as given in (\ref{decompj}). According to Lemma \ref{lemmaa4}, we have the following:

\begin{enumerate}
    \item[.] If $J_{10}J_{00} \neq \{0\}$, then either $B_i \sim_T A_4 \oplus N$ or $B_i \sim_T A_4 \oplus A_1^*\oplus N$;
    \item[.] If $J_{00}J_{01} \neq \{0\}$, then either $B_i \sim_T A_4^* \oplus N$ or $B_i \sim_T A_4^* \oplus A_1\oplus N$.
\end{enumerate}

Therefore, we may now assume that $J_{10}J_{00} = J_{00}J_{01} = \{0\}$. By Lemma \ref{lemmaa5}, if $J_{01}J_{10} \neq \{0\}$, then
$$
B_i \sim_T A_5 \oplus N.
$$

Hence, we consider the case where $J_{01}J_{10} = \{0\}$. In this case, it follows from Lemma \ref{lemmaa6} that, if $J_{10}J_{01} \neq \{0\}$, then
$$
B_i \sim_T A_6 \oplus N.
$$

Thus, we now assume
$$
J_{10}J_{00} = J_{00}J_{01} = J_{01}J_{10} = J_{10}J_{01} = \{0\}.
$$

Note that, under these conditions, $B_i = B \oplus J_{00}$, where $B = F + J_{10} + J_{01} + J_{11}$. Observe that, by Lemma \ref{lemmag4}, if $\mathcal{G}_4 \in \textnormal{var}(B_i)$, then $B_i$ is PI-equivalent to one of the following algebras:
$$
\mathcal{G}_4 \oplus N\quad \mbox{ or }\quad\mathcal{G}_4 \oplus A_1\oplus N\quad\mbox{ or }\quad\mathcal{G}_4 \oplus A_1^*\oplus N.
$$

Therefore, we also assume that $\mathcal{G}_4 \notin \textnormal{var}(B_i)$, and hence $\mathcal{G}_4 \notin \textnormal{var}(B)$. Since $Q \notin \textnormal{var}(B)$ for all $Q \in \mathcal{I}_6$, Lemma \ref{2}(2) implies that $B \in \textnormal{var}(UT_2)$. Using the classification given in Theorem \ref{subvarieties}, along with the fact that $Q \notin \textnormal{var}(B)$ for all $Q \in \mathcal{I}_7$, we conclude that $B$ is PI-equivalent to one of the following algebras:
$$N, C\oplus N,  A_1\oplus N,  A_1^*\oplus N,   A_1\oplus A_1^*\oplus N, A_2\oplus N, A_1 \oplus A_2 \oplus N, $$ $$ A_1^* \oplus A_2 \oplus N, A_1\oplus A_1^*\oplus A_2\oplus N, A_4 \oplus N, A_4^* \oplus N, A_4 \oplus A_1^*\oplus  N\text{ or } A_4^* \oplus A_1\oplus  N
$$
where $N$ denotes a nilpotent algebra and $C$ is a commutative non-nilpotent algebra.

Finally, we recall that $A$ is PI-equivalent to a direct sum of algebras of type $B$ with nilpotent algebras. Since certain direct sums among these algebras were excluded from $\textnormal{var}(A)$, the result follows.
\end{proof}

As a consequence, we are now able to present the classification of the varieties whose $l_n(A)=7$, for $n$ large enough. Before proceeding, we state an auxiliary result that will be essential in what follows.

\begin{proposition} \label{propo}
    Let $A$ be an algebra and let $B\in \textnormal{var}(A)$. If there exists an integer $k\geq 0$ such that $P_m\cap \textnormal{Id}(A)=P_m\cap \textnormal{Id}(B)$, for all $m\geq k$, then $A\sim_{T} B\oplus N$, where $N$ is a nilpotent algebra.
\end{proposition}

\begin{proof}
 
First, notice that since $B\in \textnormal{var}(A)$, we have $P_n \cap \textnormal{Id}(A) \subseteq P_n \cap \textnormal{Id}(B),$ for all $n.$ Moreover, it is known that there exists a nilpotent algebra \( N \) whose \( T \)-ideal of identities is given by
\[
I = \left\langle P_l \cap \textnormal{Id}(A),\ x_1 \cdots x_k \ \middle| \ 1 \leq l \leq k-1\right\rangle_{T}.
\]

By definition of $I$, it is clear that
\[
P_n \cap \textnormal{Id}(A) \subseteq P_n \cap \textnormal{Id}(B) \cap \textnormal{Id}(N) = P_n \cap \textnormal{Id}(B \oplus N), \text{ for all } n.
\]

Furthermore, we observe that any multilinear identity of degree \( n < k \) of \( N \) is a consequence of the identities \( P_l \cap \textnormal{Id}(A) \), for all \( 1 \leq l \leq k-1 \). Consequently,
\[
P_n \cap \textnormal{Id}(B) \cap \textnormal{Id}(N) \subseteq P_n \cap \textnormal{Id}(N) \subseteq P_n \cap \textnormal{Id}(A), \text{ for all } n < k.
\]

Finally, since \( N \) is nilpotent of index \( k \) and $P_m\cap \textnormal{Id}(A)=P_m\cap \textnormal{Id}(B)$, for all $m\geq k$, we obtain
\[
P_n \cap \textnormal{Id}(A) = P_n \cap \textnormal{Id}(B \oplus N), \text{for all } n \geq k.
\]

Therefore, $P_n \cap \textnormal{Id}(A) = P_n \cap \textnormal{Id}(B \oplus N)$, for all $n$, and hence \( A \sim_{T} B \oplus N. \)  
\end{proof}

Let us define the subset $\mathcal{B}$ of $\mathcal{I}$ consisting of the algebra $A$ with $l_n(A)=7$, for $n$ large enougth, i.e., $$\mathcal{B}=\{A_7, \mathcal{G}_6, A_4\oplus A_2, A_4^* \oplus A_2, A_5\oplus A_2, A_6\oplus A_2 , \mathcal{G}_4\oplus A_1 \oplus A_1^* \}$$

 \begin{corollary}
Let $A$ be an algebra over a field of characteristic zero. Then $l_{n}(A)= 7,$ $n$ large enough, if and only if $A$ is PI-equivalent to $B\oplus N$, for some $B\in \mathcal{B}$ and $N$ a nilpotent algebra.
 \end{corollary}

 \begin{proof}
    First, note that if \( A \) is \(\textnormal{T}\)-equivalent to one of the algebras $A_7, \mathcal{G}_6, A_4\oplus A_2, A_4^* \oplus A_2, A_5\oplus A_2, \linebreak  A_6\oplus A_2$ and $\mathcal{G}_4\oplus A_1 \oplus A_1^*$, then, by Lemmas \ref{A7}, \ref{G2k} and \ref{charac}, it follows that \( l_n(A) = 7 \), for $n$ large enough. 
    
    Conversely, assume that there exists $k\geq 0$ such that \( l_n(A) = 7 \), for all $n\geq k$. By Theorem \ref{maintheorem}, we must have $S\in \textnormal{var}(A)$, for some $S\in \mathcal{B}$. Notice that, for all $n\geq k$, we have $7=l_n(P)\leq l_n(A)=7$ and so, $P_n\cap \textnormal{Id}(S)=P_n\cap \textnormal{Id}(A)$, for all $n\geq k$. 
By Proposition \ref{propo}, we conclude that $A\sim_T S\oplus N$.
 \end{proof}

\black 

\section{Classifying varieties of small central colength}

In this section, we investigate the central polynomials and the central cocharacters associated with the algebras introduced above. We describe the sequences of central cocharacters and the corresponding central colengths for several of these algebras. As a consequence, we obtain the classification of the varieties whose central colengths are bounded by 2.

In the study of polynomial identities, an important role is played by the algebra of upper block triangular matrices, denoted $UT(d_1,\ldots,d_k)$. This algebra is defined as the subalgebra of $M_{d_1+\cdots+d_k}(F)$ consisting of matrices of the form  
\[
UT(d_1, \ldots, d_k) =
\left(
\begin{array}{cccc}
A_1 & A_{12} & \cdots & A_{1k} \\
& A_2 & \ddots & \vdots \\
& & \ddots & A_{k-1,k} \\
0 & & & A_k
\end{array}
\right),
\]
where $A_i \cong M_{d_i}(F)$ for $1 \leq i \leq k$, and $A_{ij} \cong M_{d_i \times d_j}(F)$ for $1 \leq i < j \leq k$.  

A fundamental result proved in \cite[Lemma~1]{GiamZai} states that, if $k>1$, this algebra admits no proper central polynomials. Therefore, we have 
\[
l_n^\delta(UT(d_1,\ldots,d_k)) = 0, \quad \text{for } k>1.
\]

This example produces an infinity list of algebras whose proper central colenths are zero. This fact indicates that a classification of varieties with respect to the sequence $l_n^\delta(A)$ seems, at least at a first stage, to be out of reach.

Here, we focus on the sequence $l_n^z(A),\, {n \geq 1}$.  Before we present the main result of this section, we consider the following results.

\begin{remark} \label{remarkcentral}
    Observe that if $Z(A)=\{0\}$ then $\chi_n^z(A)=
    \chi_n(A)$, $l_n(A)=l_n^z(A)$ and $l_n^\delta(A)=0$. In particular, $$\chi_n^z(A_1)=\chi_n^z(A_1^*)= \chi_{(n)}+\chi_{(n-1,1)}
    \quad \mbox{ and }\quad l_n^z(A_1)=l_n^z(A_1^*)=2.$$  Moreover, $l_n^z(A_1\oplus A_1^*)=3$ and $l_n^z(A_4)=l_n^z(A_4^*)=5$. 
\end{remark}

The following lemma can be verified in \cite{Costa}. 

\begin{lemma} \label{central1} For the algebras  $A_2$, $N_4$ and $\mathcal{G}_{2k}$ we have

\begin{enumerate}
    \item[1)] $\chi_n(A_2)= \chi_{(n)},$ $ l_n^z(A_2)=1\mbox{ and }l_n^\delta(A_2)=2.$
     \item[2)] $\chi_n^z(N_4)=\chi_{(n)}+\chi_{(n-1,1)}+\chi_{(n-2,1^2)},$ $ l_n^z(N_4)=3\mbox{ and } l_n^\delta(N_4)=4.$
    \item[3)] $\chi_n^z(\mathcal{G}_{2k})=\sum_{i=0}^{k-1} \chi_{(n-i,1^{2i})}, $ $l_n^z(\mathcal{G}_{2k})=k\mbox{ and } l_n^\delta(\mathcal{G}_{2k})=k+1$ , for all $k\ge 1$.
   
\end{enumerate}

\end{lemma}

\begin{lemma} For the algebras $A_5$ and $A_6$ we have
    \begin{enumerate} 
        \item[1)] $\chi_n^z(A_5)=\chi_n^z(A_6)=\chi_{(n)}+2\chi_{(n-1,1)}$ and $\chi_n^\delta(A_5)=\chi_n^\delta(A_6)=\chi_{(n-2,2)}+\chi_{(n-2,1^2)}.$
        \item[2)] $l_n^z(A_5)=l_n^z(A_6)=3$ and $l_n^\delta(A_5)=l_n^\delta(A_6)=2$.
    \end{enumerate}
\end{lemma}

\begin{proof}
First, observe that the centers of the algebras $A_5$ and $A_6$ are
given by
\[
Z(A_5)=Z(A_6)=\operatorname{span}_F\{e_{13}\}.\] We prove the statement for $A_5$ and omit the proof for $A_6$, since the
arguments are entirely analogous.

Consider first the partition $(n)$ and its highest weight vector
$f_{(n)}=x_1^n$. Evaluating at $x_1\mapsto e_{22}$, we obtain
$f_{(n)}\mapsto e_{22}\notin Z(A_5)$. This shows that
$1\le m_{(n)}^{z}\le m_{(n)}=1$, and hence $m_{(n)}^{z}=1$.

Next, consider the partition $(n-1,1)$ and the associated highest weight
vectors
\[
f^{1}_{(n-1,1)}=[x_1,x_2]x_1^{\,n-2}
\quad \text{and} \quad
f^{2}_{(n-1,1)}=x_1^{\,n-2}[x_1,x_2].
\]
Under the evaluation $x_1\mapsto e_{22}$ and
$x_2\mapsto e_{12}+e_{23}$, we have
$f^{1}_{(n-1,1)}\mapsto -e_{12}\notin Z(A_5)$ and
$f^{2}_{(n-1,1)}\mapsto e_{23}\notin Z(A_5)$. Moreover, the polynomials
$f^{1}_{(n-1,1)}$ and $f^{2}_{(n-1,1)}$ are linearly independent modulo $\textnormal{Id}^z(A_5)$. Therefore,
$2\le m_{(n-1,1)}^{z}\le m_{(n-1,1)}=2$, which implies
$m_{(n-1,1)}^{z}=2$.

Finally, for the partitions $(n-2,2)$ and $(n-2,1^2)$, one verifies that every
highest weight vector associated with them is central. Consequently, $m_{(n-2,2)}^{z}=m_{(n-2,1^2)}^{z}=0.$

Therefore, we conclude that
\[
\chi_n^{z}(A_5)=\chi_{(n)}+2\chi_{(n-1,1)}
\quad \text{and} \quad
\chi_n^{\delta}(A_5)=\chi_{(n-2,2)}+\chi_{(n-2,1^2)}.
\]

\end{proof}

\black

At this point, we remark that an analogous statement to Remark~\ref{comparacaochar} also holds for the central cocharacter. Consequently, we obtain the following result.

\begin{lemma} \label{centralcolength} For the direct sum of some of the previous algebras, we have
    \begin{enumerate}
  \item[1)] $\chi_n^z(A_2\oplus A_1)=
        \chi_n^z(A_2\oplus A_1^*)=\chi_{(n)}+\chi_{(n-1,1)}.$
    
        \item[2)] $\chi_n^z(\mathcal{G}_4\oplus A_1)=
        \chi_n^z(\mathcal{G}_4\oplus A_1^*)=\chi_{(n)}+\chi_{(n-1,1)}+\chi_{(n-1,1^2)}.$ 

        \item[3)] $l_n^z(A_2\oplus A_1)= l_n^z(A_2\oplus A_1^*)=l_n^\delta(A_2\oplus A_1)= l_n^\delta(A_2\oplus A_1^*)=2$.
 \item[4)] $l_n^z(\mathcal{G}_4\oplus A_1)= l_n^z(\mathcal{G}_4\oplus A_1^*)=l_n^\delta(\mathcal{G}_4\oplus A_1)= l_n^\delta(\mathcal{G}_4\oplus A_1^*)=3$.
    
    \end{enumerate}
\end{lemma}

Consider $$\mathcal{T}_1=\{\mathcal{G}_4\oplus A_1,\mathcal{G}_4\oplus A_1^*, N_4, \mathcal{G}_{6} \}.$$

\begin{lemma} \label{lemmag4central}  Let $A=F+J(A)$ such that $\mathcal{G}_4\in \textnormal{var}(A)$. If $Q\notin \textnormal{var}(A)$, for all $Q\in \mathcal{T}_1$, then $A\sim_T \mathcal{G}_4
\oplus N$, where $N$ is a nilpotent algebra.
\end{lemma}

\begin{proof}
Since $\mathcal{G}_4 \in \var(A)$ and
$\mathcal{G}_4 \oplus A_1,\ \mathcal{G}_4 \oplus A_1^* \notin \var(A)$, it follows
that $A_1, A_1^* \notin \var(A)$. By Lemma~\ref{1}(1)(2), we conclude that
$J_{10}=J_{01}=\{0\}$, and hence
\[
A=(F+J_{11})\oplus J_{00}.
\]
Moreover, since $N_4 \notin \var(A)$, we also have $A_7 \notin \var(A)$ and,
therefore, by \cite[Lemma~13]{Daniela},
\[
[J_{11},J_{11},J_{11}]=\{0\}.
\]
This implies that $F+J_{11} \in \var(\mathcal{G})$. Furthermore,
$\mathcal{G}_6 \notin \var(A)$, and thus, by
Theorem~\ref{subvarieties}, we obtain
\[
F+J_{11} \sim_T \mathcal{G}_4 \oplus N_1,
\]
where $N_1$ is a nilpotent algebra. Finally, since $J_{00}$ is nilpotent, the
result follows.

\end{proof}

 In order to present the main result of this section, we  introduce the following set:
\[
\mathcal{R} = \{ A_1\oplus A_1^*, A_4,A_4^*, A_5,A_6  \}\cup \mathcal{T}_1.
\]

\begin{theorem} \label{maintheorem2}
		Let $A$ be an algebra over $F$. The following conditions are equivalent:
		\begin{enumerate}
			\item[(1)]$Q\notin \textnormal{var}(A)$, for all $Q\in \mathcal{R} $.
			
			\item[(2)] $A$ is PI-equivalent to one of the following algebras: \[
N,\, C \oplus N,\, A_1 \oplus N,\, A_1^* \oplus N,\, A_2 \oplus N, \, A_2\oplus A_1 \oplus N, \, A_2 \oplus A_1^*\oplus N, \, \mathcal{G}_4\oplus N
\]
            
            where $N$ denotes a nilpotent algebra and $C$ is a commutative non-nilpotent algebra.

   \item[(3)] $l_{n}^z(A)\leq 2,$ for $n$ large enough.
		\end{enumerate}
	\end{theorem}

\begin{proof}
Since $l_n^z(B)\ge 3$ for all $B \in \mathcal{R}$, implication $(3)\Rightarrow(1)$
is immediate. Moreover, by Remark~\ref{remarkcentral} together with
Lemmas~\ref{central1} and~\ref{centralcolength}, we obtain that
$(2)\Rightarrow(3)$. Therefore, it remains to prove that $(1)\Rightarrow(2)$.

Assume that $Q \notin \var(A)$ for every $Q \in \mathcal{R}$. By
Remark~\ref{remarksubv}, this implies that $UT_2, \mathcal{G} \notin \var(A)$.
Hence, by Theorem~\ref{teokemer}, the algebra $A$ has polynomial growth and \[
A \sim_T B_1 \oplus \cdots \oplus B_m,
\] where each $B_i$ is finite dimensional and is either nilpotent or of the form
$F+J(B_i)$. If all the algebras $B_i$ are nilpotent, there is nothing to prove.

Suppose now that $B_i = F+J(B_i)$ for some index $i$. Since
$A_4, A_4^*, A_5, A_6 \notin \var(A)$, Lemma~\ref{1} yields
\[
J_{10}J_{00} = J_{00}J_{01} = J_{10}J_{01} = J_{01}J_{10} = \{0\}.
\]
Consequently, $B_i = B \oplus J_{00}$, where $B = F + J_{10} + J_{01} + J_{11}.$

If $\mathcal{G}_4 \in \var(B)$, then, since $Q \notin \var(B)$ for every
$Q \in \mathcal{T}_1$, Lemma~\ref{lemmag4central} implies that
$B \sim_T \mathcal{G}_4$. Hence,
\[
B_i \sim_T \mathcal{G}_4 \oplus N,
\]
where $N=J_{00}$ is nilpotent.

Thus, we may assume that $\mathcal{G}_4 \notin \var(B)$. In this case, since
$\mathcal{G}_4, A_7, A_8, A_8^* \notin \var(B)$, Lemma~\ref{2} implies that
$B \in \var(UT_2)$. Furthermore, since
$N_4, A_4, A_4^*, A_1 \oplus A_1^* \notin \var(B)$,
Theorem~\ref{subvarieties} establishes that $B_i$ is $\mathrm{T}$-equivalent to one of
the following algebras:
\[
N, \quad
C \oplus N, \quad
A_1 \oplus N, \quad
A_1^* \oplus N, \quad
A_2 \oplus N, \quad
A_2 \oplus A_1 \oplus N, \quad
A_2 \oplus A_1^* \oplus N,
\]
where $N$ is a nilpotent algebra and $C$ is a commutative non-nilpotent algebra.

Finally, since $A \sim_T B_1 \oplus \cdots \oplus B_m$ and certain direct sums
among the algebras listed above are excluded from $\var(A)$, we conclude that
$(1)\Rightarrow(2)$, which completes the proof.

\end{proof}

\vspace{0.5cm}
 \textbf{Acknowledgments}
\vspace{0.5cm}


\end{document}